\documentclass [oneside] {thesis} 

\usepackage{algebra-nate}
\usepackage{parsetree}
\usepackage{longtable}

\def\includecode{1}

\newcommand{\singlespacing}{\def\baselinestretch{1.0}\normalfont}
\newcommand{\origspacing}{\def\baselinestretch{1.5}\normalfont}

\newcommand{\docode}[3]{\section{\texttt{#1}} #2 
  \ifnum \includecode = 1 

  \singlespacing 
  \scriptsize \verbatimtabinput{#3} \origspacing \normalsize
  \else
  \\ \framebox{Code omitted in this copy.}
  \fi}

\newcommand{\op}{{\mathrm{op}}}

\newcommand{\rightpara}[2]{
  \begin{minipage}[b]{#1}
    \begin{flushright}
      \singlespacing
      \rule{0pt}{.5ex} \\
      #2
    \end{flushright}
  \end{minipage}
}

\begin{document}

\prelimpages

 
%
%
\Title{An Eigenspace Approach to Isotypic Projections\\[-11pt] for
  Data on Binary Trees}
\Subtitle{}

\Author{Nathaniel Eldredge}
\Advisor{Michael E. Orrison}
\Reader{Shahriar Shahriari}
\Month{May}
\Year{2003}
\titlepage

%
%
\setcounter{page}{-1}

\abstract{ The classical Fourier transform is, in essence, a way to
  take data and extract components (in the form of complex
  exponentials) which are invariant under cyclic shifts.  We consider
  a case in which the components must instead be invariant under
  automorphisms of a binary tree.  We present a technique by which a
  slightly relaxed form of the generalized Fourier transform in this
  case can eventually be computed using only simple tools from linear
  algebra, which has possible advantages in computational efficiency.
  }

%
%
\tableofcontents

\listoftables
 
%
%
\acknowledgments{
  {\narrower\noindent

  My deepest thanks go to my advisor, Prof. Michael Orrison, for his
  constant and invaluable input into this project.  I would also like
  to thank Prof. Shahriar Shahriari for acting as second reader.
  David Uminsky and Ross Richardson provided excellent advice on many
  issues, large and small, as well as general encouragement throughout
  the thesis process.  Finally, John Cloutier contributed some very
  useful advice on a talk I gave about this project.  \par} }
%

\textpages


\chapter{Introduction}

\section{The Fourier transform}

The Fourier transform is known to most scientists and engineers as a
tool for data analysis.  Given a signal, the classical Fourier
transform recovers its spectrum, which describes how the signal can be
broken into sines and cosines, or, equivalently, complex exponentials.
In the discrete case, where the signal consists of a finite number of
data points, there are well-known computational techniques for this;
most notable is the discrete fast Fourier transform (FFT) algorithm
due to Cooley and Tukey \cite{cooley-tukey}.  The FFT allows the
Fourier transform to be computed efficiently, and has become an
extremely important tool for digital signal processing in fields
ranging from physics and engineering to electronic music.

However, complex exponentials are not the only ``pieces'' into which
we might wish to decompose a signal.  The crucial feature of functions
like $e^{it}$ is that they are in a sense invariant under translation;
shifting $t$ changes the function only by a (complex) constant
multiple.  So the classical Fourier transform extracts from the signal
components which fit nicely into this translational structure.  But
there are other sorts of structure we might seek.  In fact, this
structure can be described by a group, and the idea of the Fourier
transform generalizes to cover the case of an arbitrary group.
Unfortunately, though, if computational efficiency is needed, more
work must be done.  Although the Cooley-Tukey FFT algorithm can be
generalized to some extent (see for instance
\cite{maslen:cooley-tukey}), for many groups, efficient Fourier
transform algorithms are not obvious or not known.

The Fourier transform can be thought of as a change of basis; in fact,
this is how it is often characterized in analysis.  In essence, we are
decomposing our signal space into one-dimensional subspaces, and
looking at the components of the signal that lie in these subspaces.
In some cases, it can be helpful if we relax this condition somewhat,
and decompose the signal into larger components which nevertheless
retain the important structural information we seek.  This is the idea
of isotypic projections, which we discuss in Chapter \ref{reptheory}.

\section{Eigenspaces and our approach}

One disadvantage of generalizations the Cooley-Tukey FFT is that it
relies heavily on algebraic facts about the group involved, making it
rather complicated to implement.  We shall describe an approach to
isotypic projections which relies on straightforward techniques from
linear algebra.  In particular, it can be possible to compute isotypic
projections with respect to some group via an algorithm for eigenspace
projections, if the appropriate eigenspaces are used.  The goal, then,
is to find a ``separating set'' of simultaneously diagonalizable
linear transformations whose eigenspaces are the subspaces we seek.
Chapter \ref{eigen} explains the details of this approach.  Of course,
finding such a set will necessarily require an algebraic understanding
of the group in question; but once it is found, implementation of the
projection algorithm becomes elementary.

We will be working with the automorphism groups of binary trees, to be
described in Chapter \ref{auto}.  These groups are of interest for
several tasks in signal processing; see for instance Section
\ref{haar}, as well as \cite{foote:image1} and
\cite{rockmore:wreath-fft}.  In addition, the decomposition of their
signal spaces has interesting combinatorial properties; see Section
\ref{wn-reptheory} and \cite{orrison:root}.

\section{Previous work}

The idea of using eigenspaces to compute isotypic projections was
explored in detail in \cite{orrison:thesis}.  Its inspiration comes
from the one of the myriad properties of the Jucys-Murphy elements
from the symmetric group (Section \ref{jucys-murphy} and
\cite{diaconis:murphy}), which can be applied for just this purpose.
Generalizations of these elements exist (\cite{diaconis:murphy},
\cite{russian}, \cite{ram}), but by no means have they been
generalized to all groups.

Much is known about the structure of the automorphism groups of binary
trees, and wreath product groups in general (see Section
\ref{wreath}).  Their representation theory is examined in
\cite{kerber1} and \cite{kerber2}, and more recently a combinatorial
approach to the more specific case of iterated wreath products of
cyclic groups is in \cite{orrison:root}.  Spectral analysis on these
groups has been considered in \cite{foote:image1} and
\cite{rockmore:wreath-fft}, with applications to signal and image
processing.

Computational details about the linear algebra involved have also been
considered.  \cite{orrison:thesis} gives bounds on the computational
complexity of using separating sets for several groups.
\cite{norton:thesis} presents an important optimization in the case of
the symmetric group, which is generalized in \cite{summer}.

\section{Structure of this paper}

In Chapter~\ref{reptheory} we review necessary concepts and facts from
the representation theory of finite groups.  Chapter~\ref{eigen}
discusses the ``eigenspace approach'' to isotypic projections, through
which the necessary computations for isotypic projections can be done
using simple linear algebra tools.  Chapter~\ref{auto} describes the
automorphism groups with which we shall concern ourselves.  Finally,
Chapter~\ref{sepset} constructs some separating sets for small cases.


\chapter{Representation Theory} \label{reptheory}

In this chapter we give a review of the necessary elements of
representation theory that are needed to read this paper, and lay out
the terminology and notation we shall use.  An excellent introduction
to the subject is \cite{james}.  For readers already acquainted with
representation theory, the first chapter of \cite{sagan} has a good
concise review.  \cite{dummit} is a very complete reference for any
unfamiliar concepts from group theory.

\section{Group representations} \label{groupreps}


Representation theory is, in essence, the idea of expressing abstract
algebra in terms of linear algebra.  Operations in a group are
transformed into operations in a vector space.

Let $G$ be a finite group.

\begin{definition}
  A \term{$G$-module} or \term{representation} of $G$ is a
  finite-dimensional complex vector space $V$ on which $G$ acts
  linearly.  That is, for any $g,h \in G$, $\vec{v},\vec{w} \in V$,
  and $\alpha, \beta \in \C$, we have:
  \begin{enumerate}
    \item $g\vec{v}$ is some element of $V$;
    \item If $e$ is the identity of $G$, then $e \vec{v} = \vec{v}$;
    \item $g(h\vec{v}) = (gh)\vec{v}$ (1, 2, and 3 together define an
    action of $G$ on $V$);
    \item $g(\alpha \vec{v} + \beta \vec{w}) = \alpha(g\vec{v}) +
    \beta(g \vec{w})$ (the action respects the linear structure of $V$).
  \end{enumerate}
\end{definition}

What we have, then, is that each $g \in G$ becomes a linear
transformation of $V$, and these transformations compose in the same
way that elements of $G$ multiply.  Since elements of $G$ have
inverses, so do these transformations.  So we can also think of this
correspondence as a homomorphism $\phi$ from $G$ to $GL(V)$, the set
of invertible linear transformations of $V$.  (Many authors use the
word ``representation'' to refer to this homomorphism instead of the
corresponding module.)

Once a basis for $V$ is fixed, each $\phi(g)$ can be represented as an
$n \times n$ matrix, where $n = \dim V$.  By taking the traces of
these matrices, we obtain the \term{character} $\chi$ corresponding to
$\phi$, defined by $\chi(g) = \tr \phi(g)$.  Since similar matrices
have the same trace (that is, $\tr ABA^{-1} = \tr B$), we see that the
character is independent of the basis chosen for $V$.  In fact, two
representations have the same character if and only if they are
isomorphic.  Also, $\chi(ghg^{-1}) = \tr[\phi(g)\phi(h)\phi(g)^{-1}] =
\tr \phi(h) = \chi(h)$, so that $\chi$ takes the same value on
conjugate elements of $G$.  A function $f : G \to \C$ with this
property is called a \term{class function}, since it can be considered
a function on the set of conjugacy classes of $G$.

We now consider how modules decompose.

\begin{definition}
  Let $V$ be a $G$-module.  A subspace $U \subset V$ is a
  \term{submodule} of $V$ if for each $g \in G$, $\vec{u} \in U$, we
  have $g\vec{u} \in U$ (that is, $U$ is closed under the action of
  $G$).  We say $V$ is \term{irreducible} if it has no submodules
  other than the trivial one $\{\vec{0}\}$ and itself.
\end{definition}

Irreducible modules are the most fundamental modules, as is shown by
the following central theorem.

\begin{theorem}[Maschke's Theorem]
  If $V$ is a nontrivial $G$-module, then we can write
  \begin{equation*}
    V = W_1 \oplus \dots \oplus W_k
  \end{equation*}
  where $W_1, \dots, W_k$ are irreducible $G$-modules.
\end{theorem}

In other words, every $G$-module can be decomposed into irreducible
modules.  See \cite{sagan} for a proof.

Using characters, we can say more about this decomposition.

\begin{definition}
  Let $\chi$ and $\psi$ be characters associated to representations of
  $G$.  Define the \term{inner product} $\langle \chi, \phi \rangle$
  by
  \begin{equation}
    \langle \chi, \psi \rangle = \frac{1}{\abs{G}} \sum_{g \in G}
    \chi(g) \overline{\psi(g)}.
  \end{equation}
\end{definition}

\begin{theorem} \label{character-inner}
  Let $V$ be a representation of $G$, with associated character
  $\chi$, which decomposes into irreducible submodules as
  \begin{equation*}
    V = m_1 W_1 \oplus m_2 W_2 \oplus \dots \oplus m_k W_k
  \end{equation*}
  where $m_i W_i$ denotes the direct sum of $m_i$ copies of $W_i$, and
  the $W_i$ are pairwise nonisomorphic.  If $\chi_i$ is the character
  associated with $W_i$, then
  \begin{equation}
    \langle \chi, \chi_i \rangle = m_i.
  \end{equation}
\end{theorem}

It also can be shown that irreducible characters are orthonormal with
respect to this inner product.  Using this fact, it is possible to
show that the set of irreducible characters forms a basis for the
space of all class functions on $G$.  As the dimension of this space
is equal to the number of conjugacy classes of $G$, we have the
following theorem:

\begin{theorem} \label{rep-class}
  The number of irreducible representations of $G$ is equal to the
  number of conjugacy classes of $G$.
\end{theorem}

Now, when decomposing a representation into irreducible submodules, it
may happen that some of these submodules are isomorphic to each other.
In this case, the decomposition is not unique; in fact, there are
infinitely many ways to write such a decomposition.  To remedy this
defect, we introduce the notion of an \term{isotypic} submodule, which
is simply the direct sum of one isomorphism class of irreducible
submodules of $V$.  In other words, given one irreducible submodule,
we collect together all the irreducible submodules isomorphic to it
into one larger subspace.  When this is done, the decomposition is in
fact unique.

\begin{theorem}
  If $V$ is a nontrivial $G$-module, then there is a decomposition
  \begin{equation*}
    V = W_1 \oplus \dots \oplus W_k
  \end{equation*}
  where $W_1, \dots, W_k$ are isotypic $G$-modules.  Furthermore, this
  decomposition is unique up to ordering.
\end{theorem}

\section{Examples} \label{repexamples}

Now let us see some examples of representations.

\begin{example}[The permutation representation]
  Suppose $G$ acts on a finite set $S$ with $n$ elements.  Let $\C S$
  be the set of all formal linear combinations $\sum_{i=1}^n a_i s_i$,
  where $a_i \in \C$, $s_i \in S$.  With componentwise addition and
  scalar multiplication, $\C S$ becomes a vector space.  Then we can
  make $\C S$ into a $G$-module by defining
  \begin{equation*}
    g \sum_{i=1}^n a_i s_i = \sum_{i=1}^n a_i (g s_i).
  \end{equation*}
  This is called the \term{permutation representation} of $G$
  corresponding to its action on $S$.
\end{example}

\begin{example}[The regular representation]
  If, in the previous example, we consider $G$ acting on itself by
  left multiplication, we obtain the \term{regular representation} $\C
  G$.

  Now an element of $G$ becomes a linear transformation on $\C G$.
  Then each element of $\C G$ is just a linear combination of linear
  transformations, which is again a linear transformation.  Hence each
  element of $\C G$ is itself a linear transformation of $\C G$, as
  follows:
  \begin{equation*}
    \left( \sum_i a_i g_i \right) \left(\sum_j b_j h_j \right) =
    \sum_i \sum_j a_i b_j g_i h_j.
  \end{equation*}
  It's easy to show that this puts a multiplicative structure on $\C
  G$, and for this reason $\C G$ is also called the \term{group
  algebra} or \term{group ring} of $G$.
\end{example}

The same extension works for any $G$-module $V$.  Since each element
of $G$ is a linear transformation of $V$, so is any element of $\C G$,
since a linear combination of linear transformations is again a linear
transformation:
\begin{equation*}
  \left( \sum_i a_i g_i \right) \vec{v} = \sum_i a_i (g_i \vec{v}).
\end{equation*}
For this reason, many authors prefer to think of $V$ as actually being
acted on by $\C G$ (since this action also respects the ring structure
of $\C G$), and call it instead a $\C G$-module.

The regular representation $\C G$ has the important property that it
contains \emph{every} irreducible representation.  In fact, if $\C G$
is written as a direct sum of irreducible submodules, then each
irreducible representation $W$ appears $\dim W$ times.  This yields
the identity
\begin{equation} \label{dimsum}
 \abs{G} = \dim \C G = \sum (\dim W)^2
\end{equation}
where the sum is taken over all non-isomorphic irreducible
representations $W$.

The regular representation can also be viewed as the set of all
functions $f : G \to \C$, with pointwise addition and scalar
multiplication, and the group action $(gf)(a) = f(g^{-1}a)$ for $g, a
\in G$.  This can be a useful formulation for signal processing, where
we may think of an element of $\C G$ as a signal on $\abs{G}$ points.

\section{Tensor products of representations} \label{tensor}

The tensor product allows us to construct representations of direct
products of groups.  We describe it in terms of matrices, but as we
saw in Section \ref{groupreps}, we could also describe it in terms of
$G$-modules; the two formulations are completely equivalent.  This
material comes directly from \cite{sagan} and is included here mainly
for later reference.

\begin{definition}
  Let $A = (a_{ij})$ and $B$ be matrices.  Their \term{tensor product}
  is the block matrix
  \begin{equation}
    A \tensor B = (a_{ij}B) = 
    \begin{pmatrix}
      a_{11}B & a_{12}B & \cdots \\
      a_{21}B & a_{22}B & \cdots \\
      \vdots & \vdots & \ddots
    \end{pmatrix}.
  \end{equation}
\end{definition}

Now let $G$ and $H$ be groups, with representations $\rho : G \to
GL(\C^n)$ and $\phi : H \to GL(\C^m)$ respectively.  Then their tensor
product $\rho \tensor \phi : G \times H \to \C^{nm}$, where we define
$(\rho \tensor \phi)(g,h) = \rho(g) \tensor \phi(h)$, is a
representation of $G \times H$.  It can be shown \cite{sagan} that if
$\rho$ and $\phi$ are irreducible, then so is $\rho \tensor \phi$.
Thus the representations of a direct product of two groups are
completely determined by the representations of the factors.

\section{Induced and restricted representations} \label{induce-restrict}

It is natural to ask how the subgroup structure of a group influences
its representations.  In fact, if we have $H \subgroup G$, we can
construct representations of $G$ from those of $H$, and vice versa.
We again use the matrix formulation of a representation.  This
material also comes from \cite{sagan}.

\begin{definition}
  Suppose $H \subgroup G$, and $\rho : H \to GL(\C^n)$ is a
  representation of $H$.  Let $t_1, \dots, t_k$ be a set of
  representatives for the cosets of $H$ in $G$ (where
  $k=\abs{G}/\abs{H}$).  Then the \term{induced representation} $\rho
  \uparrow_H^G : G \to GL(\C^{nk})$ maps each $g \in G$ to the block
  matrix
  \begin{equation}
    \rho \uparrow_H^G (g) =
    \begin{pmatrix}
      \rho(t_1^{-1} g t_1) & \rho(t_1^{-1} g t_2) & \cdots &
      \rho(t_1^{-1} g t_k) \\
      
      \rho(t_2^{-1} g t_1) & \rho(t_2^{-1} g t_2) & \cdots &
      \rho(t_2^{-1} g t_k) \\

      \vdots & \vdots & \ddots & \vdots \\

      \rho(t_k^{-1} g t_1) & \rho(t_k^{-1} g t_2) & \cdots &
      \rho(t_k^{-1} g t_k) \\
      
    \end{pmatrix}
  \end{equation}
  where $\rho(x)=0$ for $x \notin H$.
\end{definition}

It is shown in \cite{sagan} that this actually yields a representation
of $G$, and that any two choices of coset representatives yield
isomorphic representations, so that the induced representation is
well-defined.

The other direction is much simpler: given a representation $\rho$ for
$G$, we can produce the \term{restricted representation} $\rho
\downarrow^G_H$ of $H$ simply by taking the restriction of the map
$\rho$ to $H$.  It is obvious that this remains a representation.

We should note that the induced or restricted representations of an
irreducible representation are \emph{not} necessarily themselves
irreducible.  See \cite{clifford} for more details on when this is
true.

\section{Representation theory and the Fourier transform} \label{fourier}

Consider the case where $G = Z_n$ is the cyclic group of order $n$.
Then $\C G$ consists of $n$-dimensional complex vectors, and the
action of $G$ cyclically permutes the components.  Its irreducible
submodules are all one-dimensional (this always happens for abelian
groups \cite{dummit}), so they will be spaces of vectors which are
only scaled when their components are cyclically permuted. One such
submodule is that in which all components are equal.  Others are given
by vectors whose components vary in some sense periodically.  In fact,
each irreducible submodule is spanned by a vector of the form
\begin{equation} \label{fourier-basis}
  v_k = (1, e^{2 \pi i k / n}, e^{4 \pi i k / n}, \dots, e^{2(n-1)\pi
  i k / n})
\end{equation}
Also, since these submodules are non-isomorphic, they are in fact the
isotypic submodules of $\C G$.

So by decomposing a vector into components lying in these subspaces,
we break it into parts that look like complex exponentials.  If we
think of vectors in $\C G$ as functions $f : G \to \C$ (where the $n$
elements of $G$ can be thought of as $n$ discrete points in time),
then this looks very much like a discrete Fourier transform.  In fact,
expressing a vector in the basis $\{v_k\}_{k=1}^n$ yields the
coefficients of the classical discrete Fourier transform on $n$
points.

If we consider the group algebra as functions on the $n$ elements of
$Z_n$, then isotypic submodules will consist of functions whose values
change only by a (complex) scalar when their domains are cycled.

This notion extends to arbitrary groups $G$ (\cite{clausen},
\cite{maslen:cooley-tukey}).  Although for nonabelian groups the
isotypics will not all be $1$-dimensional, projections onto isotypic
subspaces of $\C G$ (or another $G$-module) can still yield important
information about the original vector.  To give just one example, the
case $G = S_n$ has been exploited to analyze ranked data
\cite{diaconis:rank}, such as survey results and voter preferences, in
much the same way as the case $G = Z_n$ is used to analyze time-series
data.  One particularly interesting application uses these techniques
to analyze approval voting, detecting coalitions in judicial and
legislative bodies \cite{dave:thesis}.

The problem then becomes: how should we compute isotypic projections?
Obviously, since projection onto a subspace is a linear
transformation, it has a matrix representation, so we could just
compute the projection directly.  However, the cost of doing this is
that of multiplying an $n \times n$ matrix by a vector, which in
general requires $O(n^2)$ operations, since there is no reason why
this matrix should be particularly ``nice.''

A better approach comes from an algorithm described in \cite{summer}
and \cite{orrison:thesis}.  If we can find diagonalizable linear
operators whose eigenspaces correspond well with the isotypic
submodules we seek, then we can compute isotypic projections via
eigenspace projections.  The next chapter describes how we go about
this search.

\chapter{Eigenspace Approaches to Isotypic Decomposition} \label{eigen}

As mentioned previously, isotypic projections can be computed via
eigenspace projections, given an appropriate set of linear operators.
This chapter describes the process.

\section{Separating sets} \label{separating-sets}

Let us precisely state the properties we seek in our operators.

\begin{definition}
  Let $V$ be a $G$-module which decomposes into isotypic submodules as
  $V = W_1 \oplus \dots \oplus W_k$.  A \term{separating set} for $V$
  is a set $S = \{A_1, \dots, A_m : V \to V\}$ of simultaneously
  diagonalizable linear operators on $V$ satisfying the following:
  For each isotypic submodule $W_j$ there exists a subset $S_j =
  \{A_{i_1}, A_{i_2}, \dots\} \subset S$, and a corresponding set of
  eigenspaces $\{E_{i_1}, E_{i_2}, \dots\}$, where $E_{i_1}$ is an
  eigenspace of $A_{i_1}$, and so on.  This set has the property that
  \begin{equation}
    W_j = E_{i_1} \cap E_{i_2} \cap \dots .
  \end{equation}
  That is, each isotypic can be written as an intersection of
  eigenspaces of some of the operators $A_1, \dots, A_m$.
\end{definition}

It should be clear that a separating set suffices to compute isotypic
projections.  For if $W_j = E_{i_1} \cap E_{i_2} \cap \dots$, then to
project $\vec{v}$ onto $W_j$, we need simply project it onto $E_{i_1}$
(Section \ref{projections} discusses how this can be done), project
the result onto $E_{i_2}$, and so on until we have iteratively
projected onto each eigenspace.  Then what we have is a projection
onto their intersection; namely, $W_j$.

Separating sets are considered at length in \cite{orrison:thesis}, in
which examples are given for several classes of groups.

\section{Conjugacy classes} \label{conjugacy-classes}

One particularly nice separating set for any group comes from its
conjugacy classes.  Let $C$ be a conjugacy class of $G$, and let $V$
be a $G$-module.  As each $g \in C$ is a linear operator on $V$, so is
their sum; namely, the map
\begin{equation*}
  \vec{v} \mapsto \sum_{g \in C} g \vec{v}.
\end{equation*}
This operator is called the \term{class sum} of $C$.

It can be shown (see for instance \cite{orrison:thesis}) that all
these operators are simultaneously diagonalizable, and that every
irreducible submodule $W \subset V$ is contained in an eigenspace of
the class sum of each $C$.  Furthermore, if $W$ has character $\chi$,
then the corresponding eigenvalue is given by
\begin{equation}
  \lambda(C,W) = \abs{C} \frac{\chi(C)}{\dim W}
\end{equation}
where by $\chi(C)$ we mean the value of $\chi$ at any element of
$C$ (recall that characters are class functions, so it does not matter
which element is used).

Notice that isomorphic irreducible submodules get the same eigenvalue,
and hence reside in the same eigenspace.  Thus, since an isotypic
submodule is a direct sum of isomorphic irreducible submodules, each
isotypic submodule also lies in an eigenspace of a class sum.

Thus, to build a separating set $S$ out of class sums, we only require
that for every pair $W_1, W_2$ of irreducibles, $S$ contains some
class sum $c$ whose conjugacy class $C$ has $\lambda(C, W_1) \ne
\lambda(C, W_2)$; that is, that $W_1$ and $W_2$ lie in distinct
eigenspaces of $c$.  If this is so, then when all eigenspaces
containing some $W$ are intersected, no other irreducible $W'$ can lie
in that intersection, since for some class sum $W$ and $W'$ are in
distinct eigenspaces.  It can be shown (see for instance
\cite{orrison:thesis}) that the set of \emph{all} class sums is
sufficient to form a separating set.  However, in many important cases
not all of them are actually needed, and a much smaller subset
suffices.

Notice, in fact, that the above condition is equivalent to having some
$C$ with $\frac{\chi_i(C)}{\dim W_i} \ne \frac{\chi_j(C)}{\dim W_j}$
for each $W_i, W_j$ (as the $\abs{C}$'s cancel).  And furthermore,
since the identity $\iota$ of the group must correspond to the
identity transformation on $V$, we have $\dim W = \chi(\iota)$ for
every representation $W$.  Thus these eigenvalues may be computed by
simply examining $\chi(C)$ for each class sum $C$ and irreducible
character $\chi$.  These are given by the \term{character table} of
$G$: if $G$ has irreducible characters $\chi_1, \dots, \chi_k$ and
conjugacy classes $C_1, \dots, C_k$, the character table is the $k
\times k$ matrix whose $ij$th entry is $a_{ij} = \chi_i(C_j)$.  The
order chosen for the characters and conjugacy classes is unspecified,
but usually $C_1$ is the conjugacy class of the identity.  As such,
given a character table, our eigenvalues appear as the entries of a
\term{modified character table} whose $ij$th entry is $b_{ij} = a_{ij}
/ a_{i1}$.

Now, it might appear that this makes the problem of finding a
separating set rather easy: all we have to do is generate the modified
character table, and search for a set of columns (conjugacy classes)
such that for every pair of rows (irreducibles), there is a column in
the set in which those two rows have different entries.  The number of
projections required is certainly related to the number of class sums
used, so it is reasonable to look for a separating set which is as
small as possible.  (Note, however, that the smallest separating set
does not always yield the fastest projections; see Section
\ref{cyclic-example} for an example.)  Unfortunately, we have shown
that finding a minimum-size separating set of class sums from the
modified character table is an \term{NP-complete} problem; there is
probably no algorithm to do this in polynomial time in the size of
this table.  For a further explanation and a proof of this fact, see
Appendix \ref{reduction}.

However, it is possible to approximate this problem rather well, if we
require only a near-optimal solution.  We could use the following
\term{greedy algorithm}: start by taking the conjugacy class that
distinguishes the most pairs of irreducibles.  If some pairs remain
undistinguished, take the class that distinguishes the most of the
remaining pairs.  Repeat this until all pairs are distinguished.

This greedy algorithm certainly runs in polynomial time.  It may be
possible to show that the set thus obtained is in some sense ``close''
to the size of a smallest set, thereby placing bounds on how
accurately the greedy algorithm approximates an optimal solution.  See
Appendix \ref{reduction} for further details.

\section{Isotypic projections via eigenspaces} \label{projections}

Suppose, then, that we want to compute the projections of a vector
$\vec{v}$ of dimension $n$ onto the $k$ eigenspaces of a
diagonalizable matrix $A$.  The naive approach is just to compute the
matrix of each projection operation (which is a linear transformation)
and multiply it by $\vec{v}$.  But the projection matrix may be
arbitrarily complicated, hence multiplying it by a vector requires
$O(n^2)$ operations.  If we have $k$ projections to compute, we need a
total of $O(kn^2)$ operations.  When $n$ is large (and for our
purposes it is), this is prohibitive.

However, there is an algorithm that can take advantage of nice
structure in $A$.  If $A$ is a general matrix, then multiplying it by
an arbitrary vector takes $O(n^2)$ operations.  But perhaps $A$ is
sparse, or block diagonal, or factors into smaller matrices.  In this
case, it can be multiplied by an arbitrary vector using fewer
operations.  We use $A^\op$ to denote this number of operations.

\begin{theorem} \label{arnoldi}
  Given a vector $\vec{v}$ of dimension $n$ and a diagonalizable $n
  \times n$ matrix $A$, the projections of $\vec{v}$ onto the
  $k$ eigenspaces of $A$ can be computed with $O(kA^\op + k^2 n)$
  operations.
\end{theorem}

The algorithm for this is based on a technique called the Arnoldi
iteration.  A description of the algorithm with a view to this
application can be found in \cite{summer}.

In our case, we will usually have $A^\op = O(n)$, and $k \ll n$, so
this will allow us to do projections with $O(n)$ operations.

\section{Example: Isotypic projections for the cyclic group} \label{cyclic-example}

Let us consider an example of the eigenspace method in action, in the
case of the cyclic group.  The isotypic projections we recover will
correspond to the coefficients of the discrete Fourier transform, as
described in Section \ref{fourier}.  As Fourier transform algorithms
often do, we restrict ourselves to the case where the number of
``points'' is a power of $2$.

Let $G = Z_{2^n} = \{z^0, z^1, \dots, z^{2^n-1}\}$ be the cyclic group
of order $2^n$ with generator $z$, and consider the regular
representation $\C G$.  As we saw in Section \ref{repexamples}, we can
view this as the space of functions $f : G \to \C$, which we can think
of as \term{signals} on $2^n$ points corresponding to the elements of
$G$ (in order).  In applications, this might correspond to some sort
of time-series data sampled at $2^n$ equally spaced points in time, so
we will write the elements of $\C G$ as complex $2^n$-tuples.

Since $G$ is abelian, each element is its own conjugacy class.  So by
Theorem \ref{rep-class} there are $2^n$ distinct irreducible
representations, all of which are contained in $\C G$.  It follows
that each irreducible representation has dimension $1$, and appears
only once in the decomposition of $\C G$, so in this case the isotypic
submodules of $\C G$ are exactly the irreducible submodules.

Now each element of $G$ is its own class sum, and hence a candidate
for inclusion in a separating set.  (Notice also that its matrix
representation is a $2^n \times 2^n$ permutation matrix, so in this
case $A^\op = O(2^n)$).  In fact, in this case the smallest separating
set of class sums has only one element!  The eigenspaces of the linear
transformation $z^1$ (the generator of $G$) are precisely the
irreducibles.  Thus, if we fix $\omega$ as a primitive $2^n$th root of
unity, it is not difficult to see that the eigenvalues of $z^1$ are
\begin{equation}
  \lambda_j = \omega^j
\end{equation}
and have corresponding eigenspaces
\begin{equation}
  E_j = \spanop \: \{ (1, \omega^j, \omega^{2j}, \dots,
  \omega^{(2^n-1)j}) \},
\end{equation}
for $j = 0, \dots, 2^n-1$.  Comparing (\ref{fourier-basis}), we see
that these eigenspaces are in fact the isotypic submodules of $\C G$.

However, this tiny (one element) separating set is not ideal for our
purposes.  Recall from Theorem \ref{arnoldi} that if there are $k$
eigenspaces, the time required to compute eigenspace projections is of
order $k^2$, and for this element, we have $k = 2^n$.  We can get
better efficiency by choosing more elements with fewer eigenspaces
each.

Let us consider instead the element $z^{2^{n-1}}$, which essentially
interchanges the first and last ``halves'' of the coordinates of a
vector.  It has eigenvalues $1$ and $-1$, corresponding to eigenspaces
$E_1, E_2$ where the last half is equal to or the negative of the
first half.  Computing these projections thus takes $O(2^n)$
operations.  Now consider the element $z^{2^{n-2}}$.  It has
eigenvalues $1, i, -1, -i$ (where $i = \sqrt{-1}$), and hence four
eigenspaces.  However, when we restrict $z^{2^{n-2}}$ to $E_1$, we
find that the restriction has only two eigenspaces of its own; the
same happens with $E_2$.  Furthermore, $E_1$ and $E_2$ each have
dimension $2^{n-1}$; if we do our computations in these spaces (with
an appropriate change of basis
\footnote{A significant amount of work has been swept under the rug
here.  However, it can be shown that an appropriate change of basis
can always be computed quickly.  See \cite{summer} for details.}), we
can project onto eigenspaces of the element $z^{2^{n-2}}$ with only
$O(2^{n-1})$ operations for each of $E_1, E_2$; again requiring a
total of $O(2^n)$ operations.  We have now split $\C G$ into $4$
eigenspaces.  We repeat the process with $z^{2^{n-3}}$; now we work in
each of $4$ eigenspaces, and require a total of $O(2^n)$ operations.
Continuing the process until we reach the element $z^1$, we find that
each of our $n$ steps has required $O(2^n)$ operations, for a grand
total of $O(n2^n)$.  This may seem high, but in terms of the number of
points $m=2^n$, this is only $O(m \log m)$ operations.  In fact, what
we have described is in essence an algorithm for the fast Fourier
transform.  It is equivalent \cite{orrison:thesis} to the so-called
Gentleman-Sande FFT \cite{gentleman-sande}.

By using a divide-and-conquer approach, we can achieve the same
results in far less time.  This is a common theme in Fourier transform
algorithms.

We observe in passing that that our ``better'' separating set actually
contains the first one!  The key is that we use it last, after the
space is already mostly decomposed, rather than trying to use its full
power right at the beginning of the process.  This demonstrates that
in considering a separating set, we must also consider the order in
which the elements are to be applied.  Had we used our ``better''
separating set in the reverse order, it would have been no improvement
at all.

\section{Jucys-Murphy elements} \label{jucys-murphy}

In the example of Section \ref{cyclic-example}, we saw a separating
set for which the intersections of the eigenspaces were exactly the
isotypic submodules we sought.  We do not actually need the full
strength of this condition.  It is perfectly all right for a
separating set to decompose the space more finely than the isotypics.
In particular, it suffices that the intersections of the eigenspaces
are merely all \emph{contained} in the isotypic submodules, since if
this holds, we can compute our eigenspace projections and merely add
up all the projections which lie in a single isotypic submodule.

As we noted in Section \ref{conjugacy-classes}, this will never be
necessary when our separating set consists of class sums.  However,
there are other possibilities.  For instance, we could intersect
conjugacy classes with subgroups of our group $G$, and take our
elements to be the sums of the resulting sets.  In the symmetric group
$S_n$, a particularly nice set of this kind is supplied by the
so-called Jucys-Murphy elements.

\begin{definition}
  For $2 \le j \le n$, the $j$th \term{Jucys-Murphy element} of $\C
  S_n$ is given by the sum of transpositions
  \begin{equation}
    R_j = (1\:j) + (2\:j) + \dots + ((j-1) \: j).
  \end{equation}
\end{definition}

Recalling \cite{dummit} that two elements of $S_n$ are conjugate if
and only if they have the same cycle type, we see that the set of all
transpositions in $S_n$ form a conjugacy class $K_{(2)}$.  Furthermore,
for $j \le n$, we have $S_j \le S_n$ in a very natural way (if $S_n$
is the group of permutations of $\{1, \dots, n\}$, then consider $S_j$
as the subgroup consisting of permutations which fix $j+1, j+2, \dots,
n$).  Then $R_j$ is simply the sum of the subset $(K_{(2)} \cap S_j) -
S_{j-1}$ of $S_n$.

It can be shown \cite{orrison:thesis} that we can get the separating
set we desire by taking all of the Jucys-Murphy elements; namely, the
set $\{R_j \where{2 \le j \le n}\}$.  Moreover, their matrix
representations in the standard basis for $\C S_n$ are quite simple:
there are only $j$ nonzero entries in each row and column (and these
are $1$s).  As $j \le n \ll n! = \dim \C S_n$, these matrices are
computationally very inexpensive to multiply, which is desirable in
view of Theorem \ref{arnoldi}.

A further, extremely useful property of the Jucys-Murphy elements
appears when we consider them as acting on $\C S_n$ not only by
\emph{left} multiplication, but also by \emph{right} multiplication.
Then the right action of $R_j$ gives rise to a different linear
transformation on $\C S_n$, which we may call $R_j'$.  As mentioned by
\cite{orrison:thesis} (with reference to \cite{drozd} and
\cite{murphy}), if we include these right-acting elements in our set
(to obtain $\{R_j, R_j' \where{2 \le j \le n}\}$), the resulting
decomposition is so fine that all of the (nontrivial) eigenspace
intersections are $1$-dimensional.  As such, computing the projections
of a vector onto these intersections amounts to a change of
basis---much as the Fourier transform in the $Z_n$ case.  In fact,
what we recover is exactly the discrete Fourier transform on $S_n$.

Much work has been done on generalizing these elements to other
groups; see \cite{diaconis:murphy}, \cite{russian}, and \cite{ram}.
However, we are not aware of any analogue of the Jucys-Murphy elements
for the groups in which we shall be interested (see
Chapter~\ref{auto}).


\chapter{Automorphism Groups of Binary Trees} \label{auto}

\section{Binary trees} \label{bintree}

For the rest of this thesis, we shall be interested in the following
class of groups.

\begin{definition}
  $W_n$ is the group of all automorphisms (or symmetries) of a
  complete binary tree $T_n$ of height $n+1$.
\end{definition}

As seen in the following example, an automorphism of such a tree
corresponds to a permutation of its leaves, and this correspondence is
one-to-one.  In this sense, $W_n$ is isomorphic to a subgroup of the
symmetric group $S_{2^n}$.

\begin{example}
  Consider the following tree $T_3$:
  
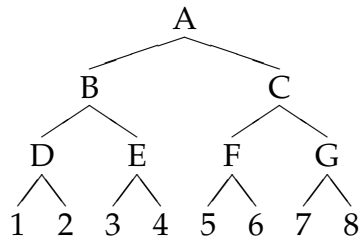
\begin{center}
    \begin{parsetree}
     ( .A. ( .B. ( .D. .1. .2. ) ( .E. .3. .4. ) )
      ( .C. ( .F. .5. .6. ) ( .G. .7. .8. ) ) )
    \end{parsetree}
  \end{center}
  We can see that the permutation $(1\:2)$ (written in cycle notation)
  corresponds to an automorphism of $T_3$, but that $(1\:3)$ does not.

  We can get automorphisms of $G$ by swapping the subtrees of any of
  the non-leaf nodes A--G, and all automorphisms can be obtained
  by composing these.  In fact, the group $W_3$ of all automorphisms
  of $T_3$ is generated by the elements $(1\:2)$, $(1\:3)(2\:4)$,
  $(1\:5)(2\:6)(3\:7)(4\:8)$, which correspond respectively to swaps at
  D, B and A.
\end{example}

The recursive structure of these groups is obvious.  In particular,
notice that $T_n$ consists of two copies of $T_{n-1}$ under a root
node.  Any automorphism of $T_n$ can be written as a product
(composition) of an automorphism of the left-hand $T_{n-1}$, an
automorphism of the right-hand $T_{n-1}$, and possibly a swap of the
two copies.  Thus we see that $\abs{W_n} = 2\abs{W_{n-1}}^2$, and
since $\abs{W_1}=2$, we have by induction that $\abs{W_n} =
2^{(2^n-1)}$.  This extremely rapid growth of the group with respect
to $n$ is the fundamental cause of computational difficulties: the
group is just too big.

\section{Wreath products} \label{wreath}

A nice description of $W_n$ can be given in terms of wreath products,
which we now define.

\begin{definition}
  Let $G$ be a finite group, and let $H \subgroup S_n$ be a
  permutation group.  Let $G^n = G \times \dots \times G$ ($n$ times)
  be the set of ordered $n$-tuples of elements of $G$.  The
  \term{wreath product} $G \wr H$ of $G$ with $H$ is the set $G^n
  \times H$ with the following multiplication:
  \begin{align}
    (g, \sigma)(h, \pi) &= (gh^\sigma, \sigma \pi)  \\
    &= ((g_1 h_{\sigma^{-1}(1)}, \dots, g_n h_{\sigma^{-1}(n)}),
    \sigma \pi)
  \end{align}
  where $g = (g_1, \dots, g_n), h = (h_1, \dots, h_n)$ are in $G^n$, and
  $\sigma$ and $\pi$ are in $H$.
\end{definition}

To understand this, imagine that the components of $h$ are ``twisted''
by $\sigma$ before being multiplied by $g$.

It is easy to show that $G \wr H$ is a group under this
multiplication.  It is also not hard to show that the wreath product
is associative, but generally not commutative.  Furthermore, it is
apparent that the wreath product is a semidirect product $G^n \rtimes
H$.

In our case, we take $G = W_{n-1}$ and $H = Z_2$.  Then $W_{n-1} \wr Z_2$
consists of two copies of $W_{n-1}$ which can be ``twisted'' together.
These copies of $W_{n-1}$ correspond to automorphisms of two copies of
$T_{n-1}$, and the twisting corresponds to the possibility of
interchanging the copies of $T_{n-1}$, as if they were subtrees of a
root node.  In fact, what we obtain is all automorphisms of $T_n$, and
we have $W_n = W_{n-1} \wr Z_2$.  Since $W_1 = Z_2$, we can write
\begin{equation}
  W_n = Z_2 \wr \dots \wr Z_2 \quad \text{($n$ times)}.
\end{equation}

In general, the group of all automorphisms of a complete regularly
branching $r$-ary tree of height $n+1$ is given by $S_r \wr \dots \wr
S_r$ ($n$ times).  By choosing at each step some subgroup of $S_r$, we
obtain a more restricted set of automorphisms.  \cite{foote:image1}
describes applications of the group $W_{n,r} = Z_r \wr \dots \wr Z_r$
($n$ times), with particular interest in the case $r=4$, in which a
``wreath product transform'' for image processing can be obtained.

\section{Representation theory} \label{wn-reptheory}

Given the recursive structure of $W_n$, it should come as no surprise
that its representations arise recursively.  This section follows a
construction from \cite{orrison:root}, which generalizes to wreath
products of arbitrary cyclic groups; a discussion of the
representation theory of wreath product groups in general may be found
in \cite{kerber1}.  The process of constructing representations of a
semidirect product is the purview of Clifford theory \cite{clifford};
a good source on the subject is \cite{karpilovsky}.  Tensor products
$\rho \tensor \phi$ of representations are defined in Section
\ref{tensor}; induced representations $\rho \uparrow_H^G$ are defined
in Section \ref{induce-restrict}.

We start with the irreducible representations $\{\rho_i\}$ of
$W_{n-1}$, and consider the normal subgroup $W_{n-1} \times W_{n-1}
\normal W_n$ which corresponds to automorphisms of $T_n$ which do not
swap the right and left subtrees of the root.  As shown in Section
\ref{tensor}, the irreducible representations of $W_{n-1} \times
W_{n-1}$ are of the form $\rho_i \tensor \rho_j$.

If $i=j$, then $\rho_i \tensor \rho_i$ is actually an irreducible
representation of $W_n$ which simply disregards the swap at the root.
To take this swap into account, we tensor an irreducible
representation of $Z_2$.  There are two of these---the trivial
representation $\phi_0$ and the alternating representation
$\phi_1$---and thus we obtain irreducible representations for $W_n$ of
the form $\rho_i \tensor \rho_i \tensor \phi_k$.

If $i \ne j$, then $\rho_i \tensor \rho_j$ is not a representation of
$W_n$.  However, since $W_{n-1} \times W_{n-1} \subgroup W_n$, we can
induce this representation of the former to a representation of the
latter, as described in Section \ref{induce-restrict}.  It can be
shown that the resulting representation $\rho_i \tensor \rho_j
\uparrow_{W_{n-1} \times W_{n-1}}^{W_n}$ of $W_n$ is irreducible.
Furthermore, the representations which arise from $\rho_i \tensor
\rho_j$ and $\rho_j \tensor \rho_i$ are isomorphic to one another, so
that the ordering of $i$ and $j$ can be ignored.

We summarize this construction in the following theorem, which is
proved in \cite{orrison:root}.

\begin{theorem} \label{repsthm}
  Suppose $\{\rho_i\}$ are all the irreducible representations of
  $W_{n-1}$.  Let $\phi_0$ and $\phi_1$ be the trivial and alternating
  representations of $Z_2$.  Then every irreducible representation of
  $W_n$ takes exactly one of the following forms:
  \begin{enumerate}
    \item $\rho_i \tensor \rho_i \tensor \phi_k$, or
    \item $\rho_i \tensor \rho_j \uparrow_{W_{n-1}\times
    W_{n-1}}^{W_n}$, for $i < j$.
  \end{enumerate}
\end{theorem}

This gives us a nice recursive way to index irreducible
representations of $W_n$, using labeled trees of height $n$.  For $W_1
\iso Z_2$, there are only two irreducible representations, the trivial
representation $\rho_0$ and the alternating representation $\rho_1$.
To these, we associate trees of height $1$, whose single node is
labeled $0$ for trivial or $1$ for alternating.  Otherwise, an
irreducible representation of $W_n$ is associated with a labeled tree
consisting of a root and two subtrees, each of which correspond to an
irreducible representation of $W_{n-1}$.  If the two subtrees are the
same, the root may be labeled with a $0$ or a $1$ (this corresponds to
the first case of Theorem \ref{repsthm}).  Otherwise, if they are
different, the root must be labeled $0$ (this corresponds to the
second case).  Notice that isomorphic trees yield isomorphic
representations.  \cite{orrison:root} calls these trees
\term{$2$-trees} (a special case of $r$-trees for iterated wreath
products of any $Z_r$), and we shall follow this terminology.

\begin{example}
The $2$-tree
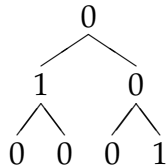
\begin{center}
  \begin{parsetree}
    ( .0. ( .1. .0. .0. ) ( .0. .0. .1. ) )
  \end{parsetree}
\end{center}
corresponds to the following irreducible representation for $W_3$:
\begin{equation*}
  (\rho_0 \tensor \rho_0 \tensor \phi_1) \tensor
  (\rho_0 \tensor \rho_1 \uparrow_{W_1 \times W_1}^{W_2})
  \uparrow_{W_2 \times W_2}^{W_3}.
\end{equation*}
\end{example}

This bijection between $2$-trees and irreducible representations lets
us count the irreducible representations of $W_n$.  In fact, the
following recurrence is easy to see:

\begin{theorem} \label{irred-count}
  Let $k_n$ be the number of irreducible representations of $W_n$.
  Then $k_1 = 2$, and
  \begin{equation} \label{recurrence}
    k_{n+1} = 2k_n + \binom{k_n}{2} = \frac{k_n^2+3k_n}{2}.
  \end{equation}
\end{theorem}

\begin{proof}
  We count the $2$-trees of height $n+1$.  Given an $2$-tree, suppose
  its root is labeled with $a = 0$ or $1$, and the two subtrees of the
  root are the $2$-trees $A$ and $B$ of height $n$.  We have the
  following possibilities:
  \begin{enumerate}
    \item $A = B$; that is, the two subtrees are equivalent.  Then
    there are $k_n$ choices for the subtree $A=B$, and the root may be
    labeled with either a $0$ or a $1$.  This gives us $2k_n$
    possibilities.

    \item $A \ne B$.  The root is then forced to be labeled with a
    $0$.  Since order does not matter, there are
    $\binom{k_n}{2}$ choices for $A,B$.
  \end{enumerate}
  As these cases are disjoint and cover every $2$-tree of height
  $n+1$, (\ref{recurrence}) follows.
\end{proof}

We are not aware of any closed-form solution of this recurrence.
However, since for $n \ge 2$ we have $k_n > 3$, it follows that
$k_{n+1} < k_n^2$, and so $k_n \le 2^{2^{n-1}}$.  In fact, it seems
empirically that the growth is rather slower than this.


\section{Permutation representation and Haar wavelets} \label{haar}

As we saw, each element of $W_n$ induces a permutation on the $2^n$
leaves of $T_n$ (see Section \ref{bintree}).  This gives rise to a
group action of $W_n$ on the set of leaves $L$, and as discussed in
Section \ref{repexamples}, this action in turn gives rise to a
permutation representation of $W_n$.  The permutation representation
can be thought of as the vector space $\C L$ of linear combinations of
$L$, or alternatively as complex-valued functions $f : L \to \C$.  In
either case it can be viewed as a space of signals, and the structure
of the representation gives us a way to decompose these signals.  In
particular, we are interested in their isotypic projections.

It can be shown \cite{foote:image1} that these projections correspond
to the $1$-D discrete Haar wavelet transform of the signal.  This
transform essentially involves decomposing the signal as a sum of
smaller and smaller square waves.  One particular advantage of the
Haar wavelet transform is that it can very effectively ``zoom in'' on
short-term, transient parts of the signal, without losing the signal's
overall shape.  Further details can be found in \cite{wavelet}, which
discusses applications including compression and denoising of signals.

\section{Conjugacy classes} \label{wreath-conjugacy-classes}


In the course of this work, we found it useful to explicitly derive
some results about the conjugacy classes of $G \wr Z_2$, of which $W_n
= W_{n-1} \wr Z_2$ is a special case.  We record them here.

Let $G$ be a finite group with identity $\iota$ and $Z_2 = \{0,1\}$ be
the cyclic group of order $2$.  We write elements of $G \wr Z_2$ as
ordered triples $(a,b,z)$ where $a,b \in G$ and $z \in Z_2$.

For readers who prefer to think of trees, think of $G = W_{n-1}$.
Then $(a,b,z)$ corresponds to an automorphism of $T_n$ constructed as
follows:
\begin{enumerate}
  \item Apply the automorphism $a$ to the left-hand subtree of the
  root (this subtree is a copy of $T_{n-1}$).

  \item Apply the automorphism $b$ to the right-hand subtree.

  \item If $z$ is the generator of $Z_2$, exchange the two subtrees;
  if $z$ is the identity of $Z_2$, do nothing.
\end{enumerate}

Recall that in general for wreath product groups $G \wr H$ where $H
\subgroup S_n$, multiplication is given by $(a,\pi) \cdot (b,\sigma) =
(ab^\pi, \pi \sigma)$ (where $a,b \in G^n$ and $b^\pi$ denotes
permuting the ``coordinates'' of $b$ according to $\pi$).  It follows
that inverses are given by $(a, \pi)^{-1} = ((a^{-1})^{\pi^{-1}},
\pi^{-1})$, and conjugation by $(a,\pi)(b,\sigma)(a,\pi)^{-1} = (a
b^\pi (a^{-1})^\sigma, \pi \sigma \pi^{-1})$.  Notice that when $H$ is
abelian (as in our case), the last coordinate of an element is
unchanged by conjugation.

Let $\sim$ denote the conjugacy relation (i.e. $a \sim b$ if $a =
xbx^{-1}$ for some $x$).

\begin{proposition} \label{case0}
  $(a,b,0) \sim (c,d,0)$ if and only if $a \sim c$ and $b \sim d$, or
  $a \sim d$ and $b \sim c$.
\end{proposition}

\begin{proof}
  Suppose $(a,b,0) \sim (c,d,0)$.  There are two cases:
  \begin{enumerate}
    \item $(x,y,0)(a,b,0)(x,y,0)^{-1} = (c,d,0)$.  Expanding,
    $(xax^{-1}, yby^{-1},0) = (c,d,0)$.  Thus $xax^{-1} = c$,
    $yby^{-1} = d$, and we have $a \sim c$ and $b \sim d$.

    \item $(x,y,1)(a,b,0)(x,y,1)^{-1} = (c,d,0)$.  Expanding,
    $(xbx^{-1}, yay^{-1}, 0) = (c,d,0)$.  Thus $xbx^{-1} = c$,
    $yay^{-1} = d$, and we have $b \sim c$ and $a \sim d$.
  \end{enumerate}
  Note that each step is reversible, so the converse is also
  established.
\end{proof}

\begin{proposition} \label{case1}
  $(a,\iota,1) \sim (c,d,1)$ if and only if $a \sim cd$.
\end{proposition}

\begin{proof}
  ($\Rightarrow$) Suppose $(a,\iota,1) \sim (c,d,1)$.  There are two
  cases:

  \begin{enumerate}
    \item $(x,y,0)(a,\iota,1)(x,y,0)^{-1} = (c,d,1)$.  Expanding,
    $(xay^{-1}, yx^{-1}, 1) = (c,d,1)$.  Thus $c = xay^{-1},
    d=yx^{-1}$, and then $cd=xax^{-1}$, so that $a \sim cd$.

    \item $(x,y,1)(a,\iota,1)(x,y,1)^{-1} = (c,d,1)$.  Expanding,
    $(xy^{-1}, yax^{-1}, 1) = (c,d,1)$.  Thus $c = xy^{-1}, d =
    yax^{-1}$, and then $cd=xax^{-1}$, so that again $a \sim cd$.
  \end{enumerate}

  ($\Leftarrow$) Suppose $a \sim cd$, so that $a = z(cd)z^{-1}$ for
  some $z \in G$.  Let $x=z^{-1}$, $y=dz^{-1}$.  Then
  \begin{align*}
    (x,y,0)(a,\iota,1)(x,y,0)^{-1} 
    &= (xay^{-1}, yx^{-1}, 1) \\
    &= ((z^{-1})(zcdz^{-1})(zd^{-1}), (dz^{-1})z, 1) \\
    &= (c, d, 1).
  \end{align*}
  Thus $(a,\iota,1) \sim (c,d,1)$.
\end{proof}

\begin{corollary} \label{class-reps}
  Every element of $G \wr Z_2$ is conjugate to an element of the form
  $(a,b,0)$ or $(a, \iota, 1)$ (and never both).
\end{corollary}

This gives us a way to count the conjugacy classes of $G \wr Z_2$.

\begin{proposition} \label{class-count}
  If $G$ has $k$ conjugacy classes, then $G \wr Z_2$ has $\binom{k}{2}
  + 2k$ conjugacy classes.
\end{proposition}

\begin{proof}
  Let $\{c_1 = \iota, c_2, \dots, c_k\}$ be a complete set of
  representatives for the conjugacy classes of $G$.  Given Corollary
  \ref{class-reps} and the fact that $(c_i, c_j, 0) \sim (c_j, c_i,
  0)$ (from Proposition \ref{case0}), we find that a complete set of
  representatives for the conjugacy classes of $G \wr Z_2$ is given by
  \begin{equation}
    \{ (c_i, c_j, 0) \where{i < j} \} \cup \{ (c_i, c_i,
    0) \} \cup \{(c_i, \iota, 1)\}.
  \end{equation}
  The first set contains $\binom{k}{2}$ elements, while the second and
  third contain $k$ elements each.  As the union is obviously
  disjoint, the conclusion follows.
\end{proof}

In the case $G = W_{n-1}$, the fact that the number of conjugacy
classes equals the number of irreducible representations means that
Proposition \ref{class-count} gives an alternate proof of Theorem
\ref{irred-count}.  On the other hand, comparing the set of class
representatives given in Proposition \ref{class-count} with the
$2$-tree construction given in Section \ref{wn-reptheory} shows us
that $2$-trees correspond in a very natural way with conjugacy
classes.  So there are natural bijections between $2$-trees,
irreducible representations, and conjugacy classes.

We can now compute the sizes of the conjugacy classes of $G \wr Z_2$.
For $g \in G$, let $C_g$ denote the conjugacy class of $g$ in $G$.

\begin{proposition} \label{sizes}
  \begin{enumerate}
    \item The conjugacy class of an element $(a,b,0)$, where $a \sim
    b$, has size $\abs{C_a}^2$.

    \item The conjugacy class of an element $(a,b,0)$, where $a \nsim
    b$, has size $2\abs{C_a}\abs{C_b}$.

    \item The conjugacy class of an element $(a, \iota, 1)$ has size
    $\abs{C_a}\abs{G}$.
  \end{enumerate}
\end{proposition}

\begin{proof}
  \begin{enumerate}
    \item If $(c,d,0) \sim (a,b,0)$, where $a \sim b$, by Proposition
    \ref{case0} $a \sim c$ and $d \sim b \sim a$.  Hence $c$ and $d$
    may each be any element of $C_a$, so there are a total of
    $\abs{C_a}$ elements conjugate to $(a,b,0)$.

    \item If $(c,d,0) \sim (a,b,0)$, where $a \nsim b$, by Proposition
    \ref{case0} either $a \sim c$ and $b \sim d$, or $a \sim d$ and $b
    \sim c$.  In the former case there are $\abs{C_a}$ possibilities
    for $c$ and $\abs{C_b}$ possibilities for $d$, for a total of
    $\abs{C_a}\abs{C_b}$.  In the latter case there are $\abs{C_a}$
    possibilities for $d$ and $\abs{C_b}$ possibilities for $c$, again
    for a total of $\abs{C_a}\abs{C_b}$.  Furthermore, as $a \nsim b$,
    these cases must be disjoint.  Hence there are a total of $2
    \abs{C_a} \abs{C_b}$ elements conjugate to $(a,b,0)$.

    \item If $(c,d,1) \sim (a,\iota,1)$, then by Proposition \ref{case1}
    we have $a \sim cd$.  Choose an element $a' \in C_a$; we want to
    have $cd = a'$.  Now $d$ may be any element of $G$, but then we
    are forced to have $c = a'd^{-1}$.  As there are $\abs{C_a}$
    choices for $a'$ and $\abs{G}$ choices for $d$, there must be a
    total of $\abs{C_a}\abs{G}$ elements conjugate to $(a,\iota,1)$.
  \end{enumerate}
\end{proof}


\chapter{Separating Sets} \label{sepset}


\section{Regular representations}

\subsection{Separating sets for $\C W_n$}

Using the techniques described in Section \ref{conjugacy-classes}, we
have computed separating sets of class sums for the regular
representation of $W_n$, $n \le 4$.  The relevant program code is
contained in Appendix \ref{rtreecode}.  We used
character tables generated by the GAP software package for
computational algebra \cite{GAP4}, modified as described in Section
\ref{conjugacy-classes}.  The results are summarized in Table
\ref{sep-set-sizes}.

The reason for the cutoff at $n=4$ is that $W_5$ is so large that GAP
was unable to compute its character table in a reasonable amount of
time.  Running for 12 hours on a 1.2 GHz Pentium III workstation
resulted in no apparent progress.

\begin{table}[htbp]
\begin{center}
\begin{tabular}{|c|r|r|r|l|}
  \hline
  $n$ & $\abs{W_n}$ & Irreducibles & Minimal set size & Method \\
  \hline \hline 
  1 & 2 & 2 & 1 & Trivial \\
  \hline
  2 & 8 & 5 & 2 & Inspection \\
  \hline
  3 & 128 & 20 & 4 & Brute force \\
  \hline
  4 & 32768 & 230 & $\le 9$ & Greedy algorithm \\
  \hline
  
\end{tabular}
\end{center}
\caption{Separating set sizes for the regular representation of $W_n$,
$n \le 4$}
\label{sep-set-sizes}
\end{table}

The sizes of separating sets appear to us to grow suspiciously like
powers of 2.  This suspicion would be strengthened if we were able to
find a separating set of size 8 for $W_4$; unfortunately, the greedy
algorithm only yields one of size 9, and the number of representations
is so large as not to be susceptible to brute force techniques.
Nevertheless, based on the recursive nature of these groups, we would
not be surprised if $W_{n+1}$ should have a separating set of twice the
size of the smallest one for $W_n$.  Specifically:

\begin{conjecture} \label{the-big-conjecture}
  $W_n$ has a separating set consisting of $2^n$ class sums.
\end{conjecture}

We have examined the structure of the separating sets we found in
hopes of finding a pattern, but have so far been unsuccessful.  We
list here the separating sets we have found.

\subsection{Separating sets for $\C W_2$}

There are $3$ separating sets of class sums of size $2$ for $W_2$.
They can easily be found by inspection of the character table.
Inspection also shows that there is no separating set of size $1$, and
thus the separating sets of size $2$ are minimal.

First, Table \ref{w2-conj} lists the conjugacy classes of $W_2$, in
the order used by GAP.  We list a representative for each, in cycle
notation, and the corresponding $2$-tree (see
Section~\ref{wreath-conjugacy-classes}).

\begin{table}[hbtp]
\begin{center}

\begin{tabular}{|l|l|c|}
  \hline
  Index & Representative & $2$-tree \\
  \hline

  1 & $\iota$ & \begin{parsetree} ( .0. .0. .0. ) \end{parsetree}  
\quad \quad \\ 
  \hline  

  2 & $(1\:2)$ & \begin{parsetree} ( .0. .1. .0. ) \end{parsetree} 
\quad \quad \\
  \hline  

  3 & $(1\:2)(3\:4)$ & \begin{parsetree} ( .0. .1. .1. )
  \end{parsetree} 
\quad \quad \\
  \hline

  4 & $(1\:3)(2\:4)$ & \begin{parsetree} ( .1. .0. .0. )
  \end{parsetree} 
\quad \quad \\
  \hline

  5 & $(1\:4\:2\:3)$ & \begin{parsetree} ( .1. .1. .1. )
  \end{parsetree}
  \quad \quad \\

  \hline

\end{tabular}
\caption{Conjugacy classes of $W_2$} \label{w2-conj}
\end{center}
\end{table}

Given the indexing of Table \ref{w2-conj}, Table \ref{w2-sepsets}
lists all $3$ separating sets of size $2$ for the regular
representation of $W_2$.
\begin{table}[htbp]
  \begin{center}
    \begin{tabular}{|c|c|c|}
      \hline
      $\{ 2,4 \}$ &
      $\{ 2,5 \}$ &
      $\{ 4,5 \}$ \\ \hline
    \end{tabular}
  \end{center}
  \caption{Minimal separating sets for $\C W_2$} \label{w2-sepsets}
\end{table}

\subsection{Separating sets for $\C W_3$}

Again, we begin by listing the conjugacy classes of $W_3$, in Table
\ref{w3-conj}.  Using its indexing, Table
\ref{w3-sepset} gives all $40$ separating sets of class sums of size
$4$ for the regular representation of $W_3$.  There are none of size
$3$, so these are minimal.  These separating sets were obtained by
brute force search of a character table generated by GAP \cite{GAP4}.

\begin{longtable}{|l|l|c|}
  \caption{Conjugacy classes of $W_3$ \label{w3-conj}}\\
  \hline
  Index & Representative & $2$-tree \\
  \hline
  \endfirsthead
  \caption[]{Conjugacy classes of $W_3$ (continued)}\\
  \hline
  Index & Representative & $2$-tree \\
  \hline
  \endhead
1 & $\iota$  & \begin{parsetree} ( .0.  ( .0.   .0.   .0. )  ( .0.   .0.   .0. ) ) \end{parsetree} \quad \quad \\ \hline
2 & $(1\:2)$  & \begin{parsetree} ( .0.  ( .0.   .1.   .0. )  ( .0.   .0.   .0. ) ) \end{parsetree} \quad \quad \\ \hline
3 & $(1\:2)(3\:4)$  & \begin{parsetree} ( .0.  ( .0.   .1.   .1. )  ( .0.   .0.   .0. ) ) \end{parsetree} \quad \quad \\ \hline
4 & $(1\:3)(2\:4)$  & \begin{parsetree} ( .0.  ( .1.   .0.   .0. )  ( .0.   .0.   .0. ) ) \end{parsetree} \quad \quad \\ \hline
5 & $(1\:4\:2\:3)$  & \begin{parsetree} ( .0.  ( .1.   .1.   .1. )  ( .0.   .0.   .0. ) ) \end{parsetree} \quad \quad \\ \hline
6 & $(1\:2)(5\:6)$  & \begin{parsetree} ( .0.  ( .0.   .1.   .0. )  ( .0.   .1.   .0. ) ) \end{parsetree} \quad \quad \\ \hline
7 & $(1\:2)(3\:4)(5\:6)$  & \begin{parsetree} ( .0.  ( .0.   .1.   .1. )  ( .0.   .1.   .0. ) ) \end{parsetree} \quad \quad \\ \hline
8 & $(1\:2)(5\:7)(6\:8)$  & \begin{parsetree} ( .0.  ( .0.   .1.   .0. )  ( .1.   .0.   .0. ) ) \end{parsetree} \quad \quad \\ \hline
9 & $(1\:4\:2\:3)(5\:6)$  & \begin{parsetree} ( .0.  ( .1.   .1.   .1. )  ( .0.   .1.   .0. ) ) \end{parsetree} \quad \quad \\ \hline
10 & $(1\:2)(3\:4)(5\:6)(7\:8)$  & \begin{parsetree} ( .0.  ( .0.   .1.   .1. )  ( .0.   .1.   .1. ) ) \end{parsetree} \quad \quad \\ \hline
11 & $(1\:2)(3\:4)(5\:7)(6\:8)$  & \begin{parsetree} ( .0.  ( .0.   .1.   .1. )  ( .1.   .0.   .0. ) ) \end{parsetree} \quad \quad \\ \hline
12 & $(1\:4\:2\:3)(5\:6)(7\:8)$  & \begin{parsetree} ( .0.  ( .1.   .1.   .1. )  ( .0.   .1.   .1. ) ) \end{parsetree} \quad \quad \\ \hline
13 & $(1\:3)(2\:4)(5\:7)(6\:8)$  & \begin{parsetree} ( .0.  ( .1.   .0.   .0. )  ( .1.   .0.   .0. ) ) \end{parsetree} \quad \quad \\ \hline
14 & $(1\:4\:2\:3)(5\:7)(6\:8)$  & \begin{parsetree} ( .0.  ( .1.   .1.   .1. )  ( .1.   .0.   .0. ) ) \end{parsetree} \quad \quad \\ \hline
15 & $(1\:4\:2\:3)(5\:8\:6\:7)$  & \begin{parsetree} ( .0.  ( .1.   .1.   .1. )  ( .1.   .1.   .1. ) ) \end{parsetree} \quad \quad \\ \hline
16 & $(1\:5)(2\:6)(3\:7)(4\:8)$  & \begin{parsetree} ( .1.  ( .0.   .0.   .0. )  ( .0.   .0.   .0. ) ) \end{parsetree} \quad \quad \\ \hline
17 & $(1\:6\:2\:5)(3\:7)(4\:8)$  & \begin{parsetree} ( .1.  ( .0.   .1.   .0. )  ( .0.   .1.   .0. ) ) \end{parsetree} \quad \quad \\ \hline
18 & $(1\:6\:2\:5)(3\:8\:4\:7)$  & \begin{parsetree} ( .1.  ( .0.   .1.   .1. )  ( .0.   .1.   .1. ) ) \end{parsetree} \quad \quad \\ \hline
19 & $(1\:7\:3\:5)(2\:8\:4\:6)$  & \begin{parsetree} ( .1.  ( .1.   .0.   .0. )  ( .1.   .0.   .0. ) ) \end{parsetree} \quad \quad \\ \hline
20 & $(1\:8\:4\:6\:2\:7\:3\:5)$  & \begin{parsetree} ( .1.  ( .1.   .1.   .1. )  ( .1.   .1.   .1. ) ) \end{parsetree} \quad \quad \\ \hline
\end{longtable}

\begin{table}[hbtp]
  \begin{center}
    \begin{tabular}{|c|c|c|c|}
      \hline
$\{ 2,4,5,16 \}$ &
$\{ 2,4,5,18 \}$ &
$\{ 2,4,8,16 \}$ &
$\{ 2,4,8,18 \}$ \\ \hline
$\{ 2,4,12,16 \}$ &
$\{ 2,4,12,18 \}$ &
$\{ 2,5,9,16 \}$ &
$\{ 2,5,9,18 \}$ \\ \hline
$\{ 2,5,11,16 \}$ &
$\{ 2,5,11,18 \}$ &
$\{ 2,8,11,16 \}$ &
$\{ 2,8,11,18 \}$ \\ \hline
$\{ 2,9,12,16 \}$ &
$\{ 2,9,12,18 \}$ &
$\{ 2,11,12,16 \}$ &
$\{ 2,11,12,18 \}$ \\ \hline
$\{ 4,5,7,16 \}$ &
$\{ 4,5,7,18 \}$ &
$\{ 4,5,14,16 \}$ &
$\{ 4,5,14,18 \}$ \\ \hline
$\{ 4,7,8,16 \}$ &
$\{ 4,7,8,18 \}$ &
$\{ 4,7,12,16 \}$ &
$\{ 4,7,12,18 \}$ \\ \hline
$\{ 4,12,14,16 \}$ &
$\{ 4,12,14,18 \}$ &
$\{ 5,7,9,16 \}$ &
$\{ 5,7,9,18 \}$ \\ \hline
$\{ 5,7,11,16 \}$ &
$\{ 5,7,11,18 \}$ &
$\{ 5,11,14,16 \}$ &
$\{ 5,11,14,18 \}$ \\ \hline
$\{ 7,8,11,16 \}$ &
$\{ 7,8,11,18 \}$ &
$\{ 7,9,12,16 \}$ &
$\{ 7,9,12,18 \}$ \\ \hline
$\{ 7,11,12,16 \}$ &
$\{ 7,11,12,18 \}$ &
$\{ 11,12,14,16 \}$ &
$\{ 11,12,14,18 \}$ \\ \hline
    \end{tabular}
  \end{center}
  \caption{Minimal separating sets for $\C W_3$} \label{w3-sepset}
\end{table}

\subsection{Separating set for $\C W_4$}

Table \ref{w4-sepset} gives a separating
set of size $9$ of class sums for the regular representation of $W_4$.
This was obtained by greedily searching the character table of $W_4$,
as described in Section~\ref{conjugacy-classes}.  As such, this set is
\emph{not} known to be minimal; in fact, we conjecture
(\ref{the-big-conjecture}) that it has one of size $8$.  It is also
presumably not the only set of size $9$.  Unfortunately, brute force
search is infeasible for checking this.

As $W_4$ has $230$ conjugacy classes, we do not list all of them; only
those involved in the separating set.  As before, we index them as
returned by GAP \cite{GAP4}.

\newpage

\begin{longtable}{|l|l|c|}
   \caption{Separating set for $\C W_4$ \label{w4-sepset}}\\
  \hline
  Index & Representative & $2$-tree \\
  \hline
  \endfirsthead
  \caption[]{Separating set for $\C W_4$ (continued)}\\
  \hline
  Index & Representative & $2$-tree \\
  \hline
  \endhead
4 & $( 1\: 2)( 5\: 6)$  & \begin{parsetree} ( .0.  ( .0.  ( .0.   .1.   .0. )  ( .0.   .1.   .0. ) )  ( .0.  ( .0.   .0.   .0. )  ( .0.   .0.   .0. ) ) ) \end{parsetree} \quad \quad \\ \hline

20 & $( 1\: 2)( 3\: 4)( 5\: 6)( 7\: 8)( 9\:10)(11\:12)(13\:14)$  & \begin{parsetree} ( .0.  ( .0.  ( .0.   .1.   .1. )  ( .0.   .1.   .1. ) )  ( .0.  ( .0.   .1.   .1. )  ( .0.   .1.   .0. ) ) ) \end{parsetree} \quad \quad \\ \hline

32 & $( 1\: 2)( 3\: 4)( 5\: 6)( 7\: 8)( 9\:11)(10\:12)$  & \begin{parsetree} ( .0.  ( .0.  ( .0.   .1.   .1. )  ( .0.   .1.   .1. ) )  ( .0.  ( .1.   .0.   .0. )  ( .0.   .0.   .0. ) ) ) \end{parsetree} \quad \quad \\ \hline

57 & $( 1\: 4\: 2\: 3)( 5\: 6)( 7\: 8)( 9\:10)(11\:12)(13\:14)(15\:16)$  & \begin{parsetree} ( .0.  ( .0.  ( .1.   .1.   .1. )  ( .0.   .1.   .1. ) )  ( .0.  ( .0.   .1.   .1. )  ( .0.   .1.   .1. ) ) ) \end{parsetree} \quad \quad \\ \hline

62 & $( 1\: 2)( 5\: 7)( 6\: 8)( 9\:12\:10\:11)$  & \begin{parsetree} ( .0.  ( .0.  ( .0.   .1.   .0. )  ( .1.   .0.   .0. ) )  ( .0.  ( .1.   .1.   .1. )  ( .0.   .0.   .0. ) ) ) \end{parsetree} \quad \quad \\ \hline

128 & $( 1\: 6\: 2\: 5)( 3\: 7)( 4\: 8)( 9\:10)(11\:12)$  & \begin{parsetree} ( .0.  ( .1.  ( .0.   .1.   .0. )  ( .0.   .1.   .0. ) )  ( .0.  ( .0.   .1.   .1. )  ( .0.   .0.   .0. ) ) ) \end{parsetree} \quad \quad \\ \hline

133 & $( 1\: 6\: 2\: 5)( 3\: 8\: 4\: 7)$  & \begin{parsetree} ( .0.  ( .1.  ( .0.   .1.   .1. )  ( .0.   .1.   .1. ) )  ( .0.  ( .0.   .0.   .0. )  ( .0.   .0.   .0. ) ) ) \end{parsetree} \quad \quad \\ \hline

158 & $( 1\: 4\: 2\: 3)( 5\: 6)( 9\:14\:10\:13)(11\:15)(12\:16)$  & \begin{parsetree} ( .0.  ( .0.  ( .1.   .1.   .1. )  ( .0.   .1.   .0. ) )  ( .1.  ( .0.   .1.   .0. )  ( .0.   .1.   .0. ) ) ) \end{parsetree} \quad \quad \\ \hline

216 & $( 1\:10\: 2\: 9)( 3\:12\: 4\:11)( 5\:14\: 6\:13)( 7\:16\: 8\:15)$  & \begin{parsetree} ( .1.  ( .0.  ( .0.   .1.   .1. )  ( .0.   .1.   .1. ) )  ( .0.  ( .0.   .1.   .1. )  ( .0.   .1.   .1. ) ) ) \end{parsetree} \quad \quad \\ \hline
\end{longtable}

\section{Permutation representations} \label{perm-sepset-sec}

In Section \ref{haar} we described a permutation representation for
$W_n$, derived from the action of $W_n$ on the leaves of $T_n$.  We
shall use $V_n$ to denote this representation.  We now describe
separating sets of class sums for $V_n$.

Notice that $\dim V_n = 2^n$, so this representation is quite small
compared to the group itself.  Furthermore, the number of irreducibles
into which it decomposes is even smaller.  It can be shown in general
\cite{foote:image1} that $V_n$ is the direct sum of $n+1$
nonisomorphic irreducible submodules, and each appears in the sum with
multiplicity $1$.  Since there are fewer submodules to be separated,
separating sets are much easier to find, and much smaller.

The algorithm for finding these separating sets is much as before,
except the table we use contains only the characters for those
irreducible representations which make up $V_n$.  These can easily be
found using the inner product relation described in Theorem
\ref{character-inner}.  The character for $V_n$ is easy to compute:
since a permutation matrix contains a $1$ on the diagonal for each
fixed point, the character of an element $g$ is equal to the number of
leaves of $T_n$ which it fixes.  Once we have identified the
irreducible representations involved, which correspond to rows in the
character table, we can remove all other rows and perform a brute force
or greedy search on the remaining table.

Table \ref{perm-sepset} lists some separating sets for these
permutation representations.  Due to their small size, they were all
found using brute-force search.  As there are usually many, we do not
list them all.  We give only an example or two for each.  The numbers
refer to our previous indexing of conjugacy classes, and the following
column lists where this indexing can be found.

\begin{table}[hbtp]
  \begin{center}
    \begin{tabular}{|r|r|r|r|r|l|l|}
      \hline
      $n$ & $\dim V_n$ & \rightpara{.75in}{Number of isotypic
      subspaces} 
      & \rightpara{.6in}{Minimal set size} & \rightpara{.6in}{Number of sets}& Example & See
      table \\
      \hline \hline
      2 & 4 & 3 & 1 & 2 & $\{4\}, \{5\}$ & \ref{w2-conj} \\
      \hline
      3 & 8 & 4 & 2 & 60 & $\{2, 16\}$ & \ref{w3-conj} \\
      \hline
      4 & 16 & 5 & 2 & 1940 & $\{32, 216 \}$ & \ref{w4-sepset}  \\
      \hline \hline
      $n$ & $2^n$ & $n+1$ & ? & ? & --- & --- \\
      \hline
    \end{tabular}
  \end{center}
  \caption{Separating sets for the permutation representation of
  $W_n$, $n \le 4$} \label{perm-sepset}
\end{table}
\chapter{Conclusion} \label{conclusion}

\section{Closing remarks}

This area of mathematics has proven for us to be a very intriguing
one, uniting elements of algebra and combinatorics from ``pure''
mathematics with ``applied'' ideas from spectral analysis, algorithms,
and computational linear algebra.  We hope that our exposition and
results can spur, in some small way, further interest in the field.

\section{Future work}

Our research in this area has raised many more questions than it has
answered.  We will list several problems which we feel are worthy of
future investigation.

\subsection{Conjecture \ref{the-big-conjecture}}

A proof of Conjecture \ref{the-big-conjecture} would be very nice to
have, especially if it is constructive.  Armed with separating sets
for all $W_n$, we would immediately have an isotypic projection
algorithm.

It would also be useful to find a bound for the sizes of minimal
separating sets for the permutation representation of $W_n$, as
described in Section \ref{perm-sepset-sec}.

\subsection{Other separating sets}

We considered only separating sets consisting of class sums.  Although
these have many nice properties, they are not necessarily optimal.
Other possibilities should be considered.  In particular, thinking of
the Jucys-Murphy elements for $S_n$ (see Section \ref{jucys-murphy}),
one could consider conjugacy classes intersected with subgroups,
or some similar construction.  It is especially suggestive that $S_1
\subgroup S_2 \subgroup \dots \subgroup S_n$ and $W_1 \subgroup W_2
\subgroup \dots \subgroup W_n$ both have a strongly recursive
structure.  Also, the Jucys-Murphy elements separate representations
into finer pieces than isotypic submodules, and can even be used to
compute a genuine discrete Fourier transform; it would be very helpful
to be able to duplicate these properties for $W_n$.

\subsection{Computational bounds}

We concentrated on finding minimal-size separating sets.  However, as
we saw in Section \ref{cyclic-example}, minimal size is not always
best when we actually want to compute projections.  In fact, in order
to say anything about the computational properties of the separating
sets we found, we would have to look at how the eigenspaces of our
elements interact, considering the dimensions of their intersections
as they decompose the space.  Once computational bounds are
established for isotypic projections using our separating sets, we
could evaluate them with respect to other possible separating sets to
find one with the best computational properties.

\subsection{Greedy algorithm}

In Section \ref{separating-sets}, we described a greedy algorithm for
quickly finding separating sets from a character table.  It would be
useful to know how optimal its results are.  As mentioned in Appendix
\ref{reduction}, a greedy algorithm for a related problem (MINIMUM
TEST COLLECTION) has been well studied, and it seems likely that these
results could be brought to bear on the separating set problem.  Since
this method bounds separating set sizes, it is possible that we could
thus obtain a (nonconstructive) proof for Conjecture \ref{the-big-conjecture}.

\subsection{Extensions to iterated wreath products of cyclic groups}

We have only examined the group $W_n = Z_2 \wr \dots \wr Z_2$.  More
generally, the groups $W_{n,r} = Z_r \wr \dots \wr Z_r$ are also of
great interest.  Many of our results about the groups extend to this
case (especially when $r$ is prime); see also \cite{orrison:root} and
\cite{foote:image1}.  In particular, the case $r=4$ gives rise to a
so-called ``wreath product transform'' with useful applications in
image processing (see \cite{foote:image1}).

\appendix
\raggedbottom\sloppy


\chapter{NP-Completeness of Finding Separating Sets of Class Sums From
Character Tables} \label{reduction}

We mentioned in Section \ref{conjugacy-classes} that a separating set
of class sums for any given group can be found by examining a modified
character table.  The problem, precisely stated, is the following.


\begin{problem}{(SEPARATING SET)}
  Given an $n \times m$ matrix $(b_{ij})$ (in our case, the table of
  eigenvalues) and an integer $k$, do there exist integers $1 \le c_1,
  \dots, c_k \le m$ such that for every pair $1 \le i_1, i_2 \le n$,
  there exists $1 \le j \le k$ such that $b_{i_1 c_j} \ne b_{i_2
  c_j}$?  In other words, can we tell any two rows apart by looking
  only in columns $c_1, \dots, c_k$?
\end{problem}

\bigskip

This problem boils down to ``does there exist a separating set of size
$k$?''  If we can solve this problem efficiently, we can find a
minimal-size separating set by attempting it for ever-increasing $k$
until we find one that works.

Unfortunately, we will show that SEPARATING SET is NP-complete.  This
means that if it has a polynomial-time solution, then so does every
other problem in the class NP of problems whose solutions can be
verified in polynomial time.  This would imply that NP is equal to P,
the class of problems with polynomial-time solutions.  It is
universally believed (though not proven, remaining a famous open
conjecture) that this is not the case.  For more details on the theory
of NP-completeness, see \cite{garey-johnson}.

Our proof of this assertion is by reduction from a problem called
MINIMUM TEST COLLECTION, which we describe here.

\begin{problem}{(MINIMUM TEST COLLECTION)}
  Given a finite set $A$, a collection $C \subset \mathcal{P}(A)$, and
  a number $J \le \abs{A}$, does there exist a subcollection $C'
  \subset C$ with $\abs{C'} \le J$ such that for every pair $a_1, a_2
  \in A$, there exists a set $S \in C'$ such that $S$ contains exactly
  one of $a_1$ and $a_2$ (in other words, $\abs{\{a_1, a_2\} \cap S} =
  1$)?
\end{problem}

\bigskip

This problem can be considered as one of medical diagnosis: imagine
$A$ is a set of diseases, and $C$ is a collection of tests, each of
which will return ``positive'' in the presence of some diseases, and
``negative'' for the rest.  As such, each test may be associated with
the set of diseases for which it returns ``positive.''  The question
is, do $k$ tests suffice to narrow the diagnosis to a single disease?

It is shown in \cite{garey-johnson} (page 71) that MINIMUM TEST COLLECTION is
NP-complete.  We now show that SEPARATING SET is as well.

\begin{theorem}
  SEPARATING SET is NP-complete.
\end{theorem}

\begin{proof}
  First, it is obvious that SEPARATING SET is in NP, since a solution
  can be verified in polynomial time.  Given the integers $c_1, \dots,
  c_k$, we can test that any pair of rows $i_1, i_2$ is ``separated''
  by looking at the $k$ pairs $a_{i_1 c_j}, a_{i_2 c_j}$ for $1 \le j
  \le k$.  Repeating this for each of the $\binom{n}{2} \le n^2$ pairs of
  rows and noticing that $k \le m$, we find that verification requires
  only $O(n^2m)$ time.

  Now, suppose we have an instance $(A = \{a_1, \dots, a_n\}, C =
  \{C_1, \dots, C_m\}, J)$ of MINIMUM TEST
  COLLECTION.  We can convert it in polynomial time to an instance of
  SEPARATING SET.  Construct a $n \times m$ matrix $(b_{ij})$ where
  \begin{equation*}
    b_{ij} = 
    \begin{cases}
      1, & a_i \in C_j \\
      0, & a_i \notin C_j
    \end{cases}.
  \end{equation*}
  Set $k = J$.  We show that our instance of SEPARATING SET has a
  solution if and only if our instance of MINIMUM TEST COLLECTION did.

  Suppose that the constructed instance of SEPARATING SET has a
  solution $c_1, \dots, c_k$.  Then for any pair $1 \le i_1, i_2 \le
  n$, there is some $1 \le j \le k$ such that $b_{i_1 c_j} \ne b_{i_2
  c_j}$.  Suppose without loss of generality that $b_{i_1 c_j} = 0$
  and $b_{i_2 c_j} = 1$.  Then $a_{i_1} \in C_{c_j}$ and $a_{i_2}
  \notin C_{c_j}$.  As we can do the same for every pair $i_1, i_2$,
  it follows that the set $C' = \{C_{c_1}, \dots, C_{c_k}\}$ is of
  size $k=J$ and satisfies the conditions required by MINIMUM TEST
  COLLECTION.

  Suppose that the given instance of MINIMUM TEST COLLECTION has a
  solution $C' = \{S_{c_1}, \dots, S_{c_J}\}$.  Then for every pair
  $a_{i_1}, a_{i_2} \in A$, there exists some $S_{c_j} \in C'$ such
  that (without loss of generality) $a_{i_1} \in S_{c_j}$ but $a_{i_2}
  \notin S_{c_j}$.  Then we have $b_{i_1 c_j} = 1 \ne 0 = b_{i_2
  c_j}$.  As this is true for every pair $i_1, i_2$, the set $c_1,
  \dots, c_J$ is of size $J=k$ and satisfies the conditions required
  by SEPARATING SET.

  Thus, a polynomial-time solution for SEPARATING SET would
  immediately yield one for MINIMUM TEST COLLECTION, and thus (since
  MINIMUM TEST COLLECTION is NP-complete) for every other problem in
  NP.  Hence, since SEPARATING SET is also in NP, we have that
  SEPARATING SET is NP-complete.
\end{proof}

Notice that we do \emph{not} claim that a minimal separating set can
never be found in polynomial time.  For one thing, we have assumed
nothing about the structure of the table $(b_{ij})$.  It is possible
that when $(b_{ij})$ is actually a modified character table for some
group, it has properties which could allow us to find a separating set
more efficiently.  Also, there may be other ways to find a separating
set besides simply examining the character table.

It is mentioned in \cite{mintest} that MINIMUM TEST COLLECTION has a
greedy approximation algorithm which produces a collection
within $1+2\ln \abs{S}$ of optimal.  It is further shown that
improving upon this approximation is NP-complete.  In Section
\ref{separating-sets} we mention a greedy algorithm for SEPARATING
SET; it would be interesting to consider whether a similar bound can
be shown to apply for it.
\chapter{Program for Computing $r$-trees and Separating Sets for
$W_n$} \label{rtreecode}


\docode{sepset.cc}{Given a character table, computes separating sets.
  Contains functions for computation by either brute force or a greedy
  algorithm.}{sepset.cc}

\docode{Makefile}{Controls compilation of all other
  files.}{rtree/Makefile}

\docode{wreath.h}{Header file for wreath product-related utility
  functions.}{rtree/wreath.h}

\docode{wreath.cc}{Utility routines related to wreath
  products.}{rtree/wreath.cc}

\docode{rtree.h}{Header file for $r$-tree computation routines.}{rtree/rtree.h}

\docode{rtree.cc}{Computes all $r$-trees of desired height and
$r$.}{rtree/rtree.cc}

\docode{gen\_rtrees.cc}{Driver program to generate and print out
  rtrees.}{rtree/gen_rtrees.cc}

\docode{conjclasses.cc}{Computes conjugacy class representatives and
  sizes for corresponding $r$-trees.}{rtree/conjclasses.cc}

\docode{conjclasses\_main.cc}{Driver program to print out conjugacy
  class representatives.}{rtree/conjclasses_main.cc}




\nocite{*}   
\bibliographystyle{plain}
\bibliography{bib}

\end{document}